\documentclass[11pt]{amsart}

\usepackage[latin1]{inputenc}
\usepackage[T1]{fontenc}
\usepackage{amsmath,amsfonts,amssymb,amsthm,url,geometry,mathrsfs,color,esint,bm,mathtools,tikz}

\renewcommand{\leq}{\leqslant}
\renewcommand{\geq}{\geqslant}

\newcommand{\gG}{\mathsf{G}}
\newcommand{\gB}{\mathsf{B}}

\newcommand{\R}{\mathbf{R}}
\newcommand{\C}{\mathbf{C}}
\newcommand{\K}{\mathbf{K}}

\DeclareMathOperator{\mathspan}{\mathrm{span}}

\newcommand{\scalar}[2]{\langle #1 , #2\rangle}
\newcommand{\ketbra}[2]{| #1 \rangle \langle  #2 |}

\theoremstyle{plain}
\newtheorem{theorem}{Theorem}
\newtheorem{corollary}[theorem]{Corollary}
\newtheorem{proposition}{Proposition}

\newtheorem{lemma}[theorem]{Lemma}

\theoremstyle{remark}

\newtheorem{example}{Example}
\begin{document}
\begin{abstract}
We use the geometric concept of principal angles between subspaces to compute the noncommutative distribution of an expression involving two free projections. For example, this allows to simplify a formula by
Fevrier--Mastnak--Nica--Szpojankowski about the free Bernoulli anticommutator. We also derive economically an explicit formula for the free additive convolution of Bernoulli distributions. As a byproduct, we observe the remarkable fact that the principal angles between random half-dimensional subspaces are asymptotically distributed according to the uniform measure on $[0,\pi/2]$.
\end{abstract}

\author{Guillaume Aubrun}
\title{Principal angles between random subspaces and polynomials in two free projections}
\maketitle

\section{Principal angles}

Let $\K$ be the real or complex field. For a integer $n$, we equip $\K^n$ with its usual inner product. We set $[n]=\{1,\dots,n\}$. 
For $0 \leq k \leq n$, we denote by $\gG_{n,k}$ the \emph{Grassmann manifold} defined as the set of all $k$-dimensional subspaces of $\K^n$. Given a subspace $E \subset \K^n$, we denote by $P_E$ the orthogonal projection onto $E$.

We now introduce the concept of \emph{principal angles} which play a central role in this note. Principal angles between two subspaces generalize the notion of the angle between two lines in $\K^2$. They are defined through the following proposition.

\begin{proposition} \label{prop:principal-angles}
Let $0 \leq k,l \leq n$ and consider subspaces $E \in \gG_{n,k}$ and $F \in \gG_{n,l}$. There exist
\begin{enumerate}
    \item an orthonormal basis $(e_i)_{i \in [k]}$ of $E$,
    \item an orthonormal basis $(f_j)_{j \in [l]}$ of $F$,
    \item numbers $\theta_1 \leq \theta_2 \leq \cdots \leq \theta_{\min(k,l)}$ in $[0,\pi/2]$
\end{enumerate}
such that, for every $i \in [k]$ and $j \in [l]$
\[ \scalar{e_i}{f_j}  = \begin{cases} 0 & \textnormal{ if } i \neq j \\ \cos(\theta_i) & \textnormal{ if } i=j. \end{cases} \]
Moreover, the numbers $(\theta_i)_{i \in [\min(k,l)]}$ are uniquely defined by these conditions.
\end{proposition}

In the context of Proposition \ref{prop:principal-angles}, the numbers $(\theta_i)_{i \in [\min(k,l)]}$ are called the \emph{principal angles} between $E$ and $F$. The vectors $e_i$ and $f_j$ are sometimes called the principal vectors; they are not uniquely defined.

Principal angles are discussed in several places (see, e.g., \cite{BS,GVL}) and can be related to singular values. If $e_i$, $f_j$ and $\theta_i$ satisfy the condition of Proposition \ref{prop:principal-angles}, then
\[ P_EP_F = \left( \sum_{i \in [k]} \ketbra{e_i}{e_i} \right) \left( \sum_{j \in [l]} \ketbra{f_j}{f_j} \right) = \sum_{i \in [\min(k,l)]} \cos (\theta_i) \ketbra{e_i}{f_i} \]
is a \emph{singular value decomposition} of the operator $P_EP_F$. Conversely, one may prove Proposition \ref{prop:principal-angles} by considering a singular value decomposition of $P_EP_F$; the uniqueness of principal angles follows from the uniqueness of singular values.


We compute, on few simple examples, the spectrum of a self-adjoint expression in two orthogonal projections from the principal angles between their ranges.

\begin{proposition} \label{prop:angles-formula}
Let $E \in \gG_{n,k}$ and $F \in \gG_{n,l}$ with $k \leq l$. Let $m = \dim( E \cap F)$ and $(\theta_i)_{i \in [k-m]}$ the nonzero principal angles between $E$ and $F$. Set $P=P_E$ and $Q=P_F$. Then
 \begin{enumerate}
 \item the spectrum of $PQP$ or $QPQ$ is
  \[ \sigma(PQP) = \sigma(QPQ) = \{0_{(n-k)} \} \cup \{ \cos^2 \theta_i \} \cup \{ 1_{(m)} \},\]
\item the spectrum of $P+Q$ is
\[ \sigma(P+Q)= \{0_{(n-k-l+m)}\} \cup \{1-\cos \theta_i\} \cup \{1_{(l-k)}\} \cup\{1+\cos \theta_i\} \cup \{2_{(m)}\} ,\]
 \item the spectrum of $\imath (PQ-QP)$ is
 \[ \sigma(\imath (PQ-QP)) = \{ -\cos \theta_i \sin  \theta_i \} \cup \{0_{(n-2k+2m)} \} \cup \{ \cos \theta_i \sin  \theta_i \}, \]
 \item the spectrum of $PQ+QP$ is
 \[ \sigma(PQ+QP) = \{ \cos^2 \theta_i - \cos \theta_i \} \cup \{0_{(n-2k+m)}\} \cup \{\cos^2 \theta_i+\cos \theta_i \} \cup \{2_{(m)}\}. \]
 \end{enumerate}
In these formulas, the spectrum is counted with multiplicity, the index $i$ ranges in $[k-m]$ and the notation $\lambda_{(p)}$ stands for the eigenvalue $\lambda$ repeated $p$ times.
\end{proposition}

More generally, the spectrum of any self-adjoint polynomial in $P_E$, $P_F$ depends only on the principal angles between $E$ and $F$.

\begin{proof}
Let $(e_i)_{i \in [k]}$ and $(f_j)_{j \in [l]}$ be respective orthonormal bases of $E$ and $F$ satisfying the conclusion of Proposition \ref{prop:principal-angles}. We have $e_i=f_i$ for $i \in [m]$. Consider the orthogonal direct sum
\[ \K^n = (E \cap F) \bigoplus \left( \bigoplus_{i=m+1}^k \mathspan(e_i,f_i) \right) \bigoplus \left( \bigoplus_{j=k+1}^l \mathspan (f_j) \right) \bigoplus (E+F)^\perp.\]
The operators $P$ and $Q$ are jointly block-diagonalizable with respect to this decomposition:
\begin{itemize}
    \item the $m$-dimensional subspace $E \cap F$ is a eigenspace for $P$ and $Q$, with eigenvalue $1$,
\item for $m+1 \leq i \leq k$, the $2$-dimensional subspace $\mathspan \{e_i,f_i\}$ is stable for both $P$ and $Q$, which act respectively as the matrices
    \begin{equation} \label{eq:joint-space} \begin{pmatrix} 1 & 0 \\ 0 & 0 \end{pmatrix} \ \  \textnormal{ and } \ \ \begin{pmatrix} \cos^2 \theta_i & \cos \theta_i \sin \theta_i \\ \cos \theta_i \sin \theta_i & \sin^2 \theta_i \end{pmatrix} \end{equation}
    in the orthonormal basis $(e_i,g_i)$, where $g_i$ is defined by the formula $f_i=\cos(\theta_i) e_i + \sin(\theta_i)g_i$,
    \item for $k+1 \leq j \leq l$, the vector $f_j$ is a eigenvector for both $P$ (with eigenvalue $0$) and $Q$ (with eigenvalue $1$),
        \item the $(n-k-l+m)$-dimensional subspace $(E+F)^\perp$ is a eigenspace for $P$ and $Q$, with eigenvalue $0$.
\end{itemize}
Each result follows; the formulas involving $\theta_i$ are obtained by computing the spectrum of the corresponding polynomial in the $2 \times 2$ matrices appearing in \eqref{eq:joint-space}.
\end{proof}

For every integer $0 \leq k \leq n$, the Grassmann manifold $\gG_{n,k}$ is equipped with a unique rotation-invariant probability measure, which we call the Haar measure. A concrete way to choose a Haar distributed random element $E \in \gG_{n,k}$ is to realize $E$ as the linear span of $k$ independent standard Gaussian vectors in $\K^n$. The following lemma is well known.

\begin{lemma} \label{lemma:generic}
Consider integers $0 \leq k,l \leq n$. Let $E \in \gG_{n,k}$ and $F \in \gG_{n,l}$ be independent Haar distributed subspaces. The following holds almost surely:
\[ \dim (E + F) = \min(k+l,n), \ \ \ \dim (E \cap F) = \max(k+l-n,0). \]
Moreover, the number of nonzero principal angles between $E$ and $F$ is almost surely equal to 
$\min(k,l,n-k,n-l)$.
\end{lemma}

\begin{proof}
The first assertion is clear if we generate $E$, $F$ via Gaussian vectors. The second can then be deduced by writing $E \cap F$ as $(E^\perp + F^\perp)^\perp$ and using the fact that $E^\perp \in \gG_{n,n-k}$ and $F^\perp \in \gG_{n,n-l}$ are also independent and Haar distributed. The last point follows since the number of nonzero principal angles between $E$ and $F$ is $\min(k,l) - \dim(E \cap F)$.
\end{proof}

In this paper, we derive the limit distribution of principal angles between random subspaces using the well known connection to free probability. This question does not seem to have been discussed in the literature; we could only locate the paper \cite{AEK} which deals with the largest principal angle only. 


\section{Free probability}

We introduce very briefly some background from free probability needed for our purposes, and refer to classical references such as \cite{MS,NS,VDN} for more detail.

A $*$-probability space is a couple $(\mathcal{A},\varphi)$, where $\mathcal{A}$ is a unital complex $*$-algebra and $\varphi : \mathcal{A} \to \C$ is a linear form which is positive (i.e., $\varphi(a^*a) \geq 0$ for every $a \in \mathcal{A}$) and satisfies $\varphi(1_\mathcal{A})=1$. Given a self-adjoint element $a \in \mathcal{A}$ and a compactly supported probability measure $\mu$, we say that $\mu$ is the \emph{distribution} of $a$ if
\[ \int_\R x^k \, \mathrm{d}\mu (x) = \varphi(a^k) \]
for every integer $k \geq 0$. 

If $p \in \mathcal{A}$ is a self-adjoint projection and $\alpha = \varphi(p)$, then the distribution of $p$ is $\gB(\alpha) \coloneqq \alpha \delta_1 + (1-\alpha) \delta_0$, the Bernoulli distribution with parameter $\alpha$.

If $A$ is a self-adjoint operator on $\K^n$ with eigenvalues $\lambda_1,\dots,\lambda_n$, its \emph{empirical spectral distribution} is defined as
\[ \mu_{\mathrm{sp}}(A) = \frac{1}{n} \sum_{i=1}^n \delta_{\lambda_i}. \]
If $A$ is an orthogonal projection of rank $r$, then $\mu_{\mathrm{sp}}(A) = \gB(r/n)$.

We do not repeat here the definition of the fundamental concept of \emph{free independence} (see \cite[Chapter 5]{NS}). We rely crucially on the asymptotic freeness of  independent large-dimensional random matrices. What we need is summarized by the following proposition, which is a special case of \cite[Theorem 23.14]{NS}.

\begin{proposition} \label{prop:asymptotic-freeness}
Fix $\alpha, \beta \in [0,1]$,  and for every $n$, integers $0 \leq  k_n, l_n \leq n$ such that $\lim k_n/n = \alpha$ and $\lim l_n/n=\beta$. Suppose that
\begin{enumerate}
\item for every $n$, $E_n \in \gG_{n, k_n}$ and $F_n \in \gG_{n,l_n}$ are independent Haar distributed random subspaces,
\item $p$ and $q$ are free self-adjoint projections in a $*$-probability space, with respective distributions $\gB(\alpha)$ and $\gB(\beta)$.
\end{enumerate}
Then, for every self-adjoint polynomial in two non-commuting variables $\pi$, the sequence of probability measures
\[ \mu_{\mathrm{sp}} (\pi(P_{E_n},P_{F_n})) \]
converges towards the distribution of $\pi(p,q)$.
\end{proposition}

In this paper, the convergence of a sequence of random measures is always meant to be the weak convergence in probability.

\section{Polynomials in two free projections}

Throughout this section, we consider $p$ and $q$ to be free projections in a $*$-probability space, with distributions $\gB(\alpha)$ and $\gB(\beta)$ respectively. 

By Proposition \ref{prop:asymptotic-freeness}, the distribution of a self-adjoint polynomial in $p$, $q$ is related to the distribution of principal angles between random subspaces. In order to find the later, we consider the polynomial $pqp$. The distribution of $pqp$ is the \emph{free multiplicative convolution} of $\gB(\alpha)$ and $\gB(\beta)$ and is denoted by $\gB(\alpha) \boxtimes \gB(\beta)$. We take advantage of the fact that an explicit formula appears in the literature (see \cite[Example 3.6.7]{VDN}) 
\begin{equation} \label{eq:boxtimes} \gB(\alpha) \boxtimes \gB(\beta) = (1-\min(\alpha,\beta)) \delta_0 + \max(\alpha+\beta-1,0) \delta_1 + \mu\end{equation}
where $\mu$ is an absolutely continuous measure with density $f$ supported on $[\phi_-,\phi_+]$, with $\phi_\pm = \alpha + \beta -  2\alpha\beta \pm 2 \sqrt{\alpha\beta(1-\alpha)(1-\beta)}$, given by
\[ f(x) = \frac{\sqrt{(\phi_+-x)(x-\phi_-)}}{2\pi x(1-x)} 
.\]
The total mass of $\mu$ is $\min(\alpha,\beta,1-\alpha,1-\beta)$. In the special case $\alpha=\beta=1/2$, we have $\phi_-=0$, $\phi_+=1$ and $2\mu$ is the arcsine distribution.

We can now derive the limit distribution for principal angles between random large-dimensional subspaces. 

\begin{theorem} \label{theo:limit-principal-angles}
Fix $\alpha, \beta \in [0,1]$,  and for every $n$, integer $0 \leq k_n,l_n \leq n$ such that $\lim k_n/n = \alpha$ and $\lim l_n/n=\beta$. Set $r_n=\min(k_n,l_n,n-k_n,n-l_n)$. For each $n$, let $E_n \in \gG_{n,k_n}$, $F \in \gG_{n,l_n}$ be independent Haar-distributed random subspaces and let $(\theta^n_i)_{i \in [r_n]}$ be the nonzero principal angles between $E_n$ and $F_n$. 

As $n \to \infty$, the empirical distribution $\frac{1}{n} \sum_{i \in [r_n]} \delta_{\theta^n_i}$ converges  towards the distribution supported on $[ \arccos \sqrt{\phi_+}, \arccos \sqrt{\phi_-}]$
with density
\[ s(\theta) = \frac{\sqrt{(\phi_+ - \cos^2 \theta)(\cos^2 \theta - \phi_-)}}{{\pi \sin \theta \cos \theta}} .\]
The total mass of this distribution equals $\min(\alpha,\beta,1-\alpha,1-\beta)$.
\end{theorem}

\begin{proof}
By Lemma \ref{lemma:generic}, the number of nonzero principal angles between $E_n$ and $F_n$ is almost surely equal to $r_n$, so the random variables $(\theta^n_i)_{i \in [r_n]}$ are well-defined.
By Proposition \ref{prop:asymptotic-freeness}, the sequence $\mu_{\mathrm{sp}}(P_{E_n}P_{F_n}P_{E_n})$ converges towards $\gB(\alpha) \boxtimes \gB(\beta)$. On the other hand, we know from Proposition \ref{prop:angles-formula} that
\[ \mu_{\mathrm{sp}}(P_{E_n}P_{F_n}P_{E_n}) = \frac{n-\max(k_n,l_n)}{n}\delta_0 + \frac{\max (k_n+l_n-n,0)}{n}\delta_1 +  \frac{1}{n} \sum_{i \in [r_n]} \delta_{\cos^2 \theta_i^n}.\]
Comparing with \eqref{eq:boxtimes}, we conclude that the sequence $\frac{1}{n} \sum_{i \in [r_n]} \delta_{\cos^2 \theta_i^n}$ converges towards $\mu$, and therefore that $\frac{1}{n} \sum_{i \in [r_n]} \delta_{\theta_i^n}$ converges towards $\varphi_*\mu$, the pushforward of $\mu$ under the map $\varphi : x\mapsto \arccos \sqrt{x}$.  By the chain rule, its density of $\varphi_*\mu$ is $(f \circ \varphi^{-1})|(\varphi^{-1})'|$ and the result follows.
\end{proof}

In the special case $\alpha = \beta = 1/2$, i.e., when the involved Bernoulli distributions are fair, the situation remarkably simple. If $E$, $F$ are random lines in $\K^2$, their angle obviously follows the uniform distribution in $[0,\pi/2]$. (The analogous statement fails in higher dimension.) Surprisingly, a similar phenomenon appears at the limit.

\begin{corollary} \label{cor:principal-angles-uniform}
For every $n$, let $E_n, F_n \in \gG_{2n,n}$ be independent Haar-distributed random subspaces of dimension $n$ in $\K^{2n}$, and $(\theta_i^n)_{i \in [n]}$ the principal angles of the pair $(E_n,F_n)$. As $n \to \infty$, the empirical distribution $\frac{1}{n} \sum \delta_{\theta_i^n}$ converges towards the uniform distribution on $[0,\pi/2]$.
\end{corollary}

 We could not locate Corollary \ref{cor:principal-angles-uniform} in the literature. It would be interesting to give a direct proof of this limit theorem, given the very simple form of the limit distribution.

We can now revert our proof strategy and compute via principal angles the distribution of any self-adjoint polynomial in $p, q$. A basic case, the distribution of $p+q$, is called the \emph{free additive convolution} of $\gB(\alpha)$ and $\gB(\beta)$ and is denoted by $\gB(\alpha) \boxplus \gB(\beta)$. Although technologies to compute free additive convolutions are available (such as the $R$-transform or Boolean cumulants), their implementation is not so obvious. We could not locate the computation of $\gB(\alpha) \boxplus \gB(\beta)$ in the literature (its Cauchy transform appears in \cite[Section 4.3]{SpeicherRao} as the solution to a $4$th degree equation, but the inversion step to write explicitly the density is nontrivial). While such a computation is doable by standard methods, we believe our derivation from Theorem \ref{theo:limit-principal-angles} to be more economical.

\begin{theorem} \label{theo:boxplus-bernoulli}
For $\alpha$, $\beta$ in $[0,1]$, define
\begin{align*} 
\gamma_1 = 1 - \sqrt{\beta(1-\alpha)} - \sqrt{\alpha(1-\beta)} \\ 
\gamma_2 = 1 - \sqrt{\beta(1-\alpha)} + \sqrt{\alpha(1-\beta)} \\ 
\gamma_3 = 1 + \sqrt{\beta(1-\alpha)} - \sqrt{\alpha(1-\beta)} \\ 
\gamma_4 = 1 + \sqrt{\beta(1-\alpha)} + \sqrt{\alpha(1-\beta)}
\end{align*}
The free additive convolution of Bernoulli distributions is given by
\[ \gB(\alpha) \boxplus \gB(\beta) = \max(1-\alpha-\beta,0) \delta_0 
+ |\alpha-\beta| \delta_1 + \max(\alpha+\beta-1,0) \delta_2 + \nu\]
where $\nu$ is the absolutely continuous measure supported on 
\[ \left[\gamma_1,\min(\gamma_2,\gamma_3)\right] \cup \left[\max(\gamma_2,\gamma_3),\gamma_4\right] \]
with density given by
\[ g(t) = \frac{\sqrt{-
(t-\gamma_1)(t-\gamma_2)(t-\gamma_3)(t-\gamma_4)
}}{{\pi t (2-t)|t-1|}}. \]
The total mass of $\nu$ equals $2\min(\alpha,\beta,1-\alpha,1-\beta)$.
\end{theorem}

In the special case $\alpha=\beta=1/2$, we recover the well known fact that $\nu$ is the arcsine distribution supported on $[0,2]$.

\begin{proof}
We use the same notation as in Theorem \ref{theo:limit-principal-angles}. By Proposition \ref{prop:asymptotic-freeness}, $\gB(\alpha) \boxplus \gB(\beta)$ is the limit of the sequence $\mu_{\mathrm{sp}}(P_{E_n}+P_{F_n})$. On the other hand, we know from Proposition \ref{prop:angles-formula} that
\begin{align*} \mu_{\mathrm{sp}}(P_{E_n} + P_{F_n}) \ = & \ \frac{\max(n-k_n-l_n,0)}{n} \delta_0 + \frac{|k_n-l_n|}{n} \delta_1 + \frac{\max(k_n+l_n-n,0)}{n} \delta_2 \\
&+ \frac{1}{n} \sum_{i \in [r_n]} \delta_{1 - \cos \theta_i^n} +  \delta_{1 + \cos \theta_i^n}.
\end{align*}
Assume that $\alpha \leq \beta$ without loss of generality, so that 
$\gamma_1 = 1 -\sqrt{\phi^+}$, $\gamma_2 = 1 -\sqrt{\phi^-} \leq \gamma_3 = 1+ \sqrt{\phi^-}$ and $\gamma_4 = 1 +\sqrt{\phi^+}$. On both $[\gamma_1,\gamma_2]$ and $[\gamma_3,\gamma_4]$, the density $g$ is the given by pushforward as $(s \circ \varphi)|\varphi'|$, where $\varphi(t)=\arccos|1-t|$. The result follows.
\end{proof}

In principle, this approach can be used to compute the distribution of a general self-adjoint polynomial in two free projections as the pushforward of the measure described in Theorem \ref{theo:limit-principal-angles}. 
We give three examples below.

\begin{example}[Commutator of free projections]
We consider the polynomial $\imath(pq-qp)$, where the factor $\imath$ is introduced to make the operator self-adjoint. An immediate adaptation of the proof of Theorem \ref{theo:boxplus-bernoulli} gives that the distribution of $\imath(pq-qp)$ equals
\[\max(|2\alpha-1|,|2\beta-1|) \delta_0 + \chi^+_*\mu + \chi^-_*\mu,\]
where the last terms are the pushforward of the measure $\mu$ defined in \eqref{eq:boxtimes} by the maps $\chi_\pm(t) = \pm \sqrt{t(1-t)}$.
This result has already been obtained in \cite[p.559--560]{NSDuke}. 
We point that in the case $\alpha = \beta =1/2$, the distribution of $\imath(pq-qp)$ is the arcsine distribution supported on $[-1,1]$.
\end{example}

\begin{example}[Anticommutator of free projections]
The free anticommutator $pq+qp$ has attracted some attention in the recent years \cite{FMNS}. While one may repeat the argument given in the proof of Theorem \ref{theo:boxplus-bernoulli}, it is actually simpler to observe that $pq+qp$ can be written as $(p+q)^2-(p+q)$. It follows that its distribution is the pushforward of $\gB(\alpha) \boxplus \gB(\beta)$ under the map $t \mapsto t^2-t$. 

We detail now the computations in the special case $\alpha=\beta=1/2$. The distribution of $\gB(1/2) \boxplus \gB(1/2)$ has density
\[ h(t) = \frac{1}{\pi \sqrt{x(2-x)}}. \]
The map $t \mapsto t^2-t$ is a bijection from $[0,1/2]$ to $[-1/4,0]$ with inverse map $\psi_-(x) = \frac{1 - \sqrt{1-4x}}{2}$, and also from $[1/2,2]$ to $[-1/4,2]$ with inverse map $\psi_+(x) = \frac{1 + \sqrt{1-4x}}{2}$. Using the chain rule, we obtain the density for $pq+qp$ as
\[ u= (h\circ \psi_-)|\psi_-'| {\bf 1}_{[-1/4,2]} + (h \circ \psi_+)|\psi_+'| {\bf 1}_{[0,2]} ,\]
which can be written explicitly as
\[ u(x) = \begin{cases} \frac{\sqrt{2}}{\pi \sqrt{1-4x} \sqrt{1+2x-\sqrt{1-4x}}} & \textnormal{ if } -\frac{1}{4} \leq x \leq 0, \\ 
\frac{\sqrt{2}}{\pi \sqrt{1-4x}} \left(\frac{1}{\sqrt{1+2x-\sqrt{1-4x}}}+\frac{1}{\sqrt{1+2x+\sqrt{1-4x}}} \right) & \textnormal{ if } 0 \leq x \leq 2. \end{cases} \]
This formula is much simpler than the one from  which has been obtained in \cite[Proposition 6.11]{FMNS}. \end{example}

\begin{example}
Our last example is the more involved polynomial $p+qpq$, for which the usual free probability techniques seem unfitting. To obtain reasonable formulas, we again restrict to the case where $p$ and $q$ are free projections with distribution $\gB(1/2)$. We first compute the eigenvalues of $A+B AB$, where $A$ and $B$ are $1$-dimensional projections with angle $\theta$ between their ranges, to be
\[ \frac{1 + \cos^2 \theta \pm \sqrt{5 \cos^4 \theta - 2 \cos^2 \theta + 1}}{2} .\]
Denote this quantity by $\rho_\pm(\cos^2\theta)$.
We may describe the distribution of $p+qpq$ as the sum of pushforwards of $\gB(1/2) \boxtimes \gB(1/2)$ (i.e., the arcsine distribution) under $\rho_+$ and under $\rho_-$. After routine computations, we obtain for $p+qpq$ a distribution supported on $[0,1/5] \cup [1,2]$ and with density
\[ x \mapsto \begin{cases} 
  \frac{1}{2\pi \zeta(x) \sqrt{2x}}\left( \frac{3-5x +\zeta(x)}{\sqrt{3-3x+\zeta(x)}} +  
  \frac{3-5x -\zeta(x)}{\sqrt{3-3x-\zeta(x)}}   \right)
& \textnormal{ if } 0 < x < \frac 15\\ \frac{5x-3-\zeta(x)}{2\pi \zeta(x) \sqrt{2x}\sqrt{3-3x+\zeta(x)}}& \textnormal{ if } 1 < x \leq 2, \end{cases}\]
where $\zeta(x)=\sqrt{5x^2-6x+1}$. On Figure~\ref{figure1} we compare this limit distribution with its approximation by two half-rank projections in $\R^{2000}.$
\begin{figure}[htpb]
\centerline{\includegraphics[scale=.35]{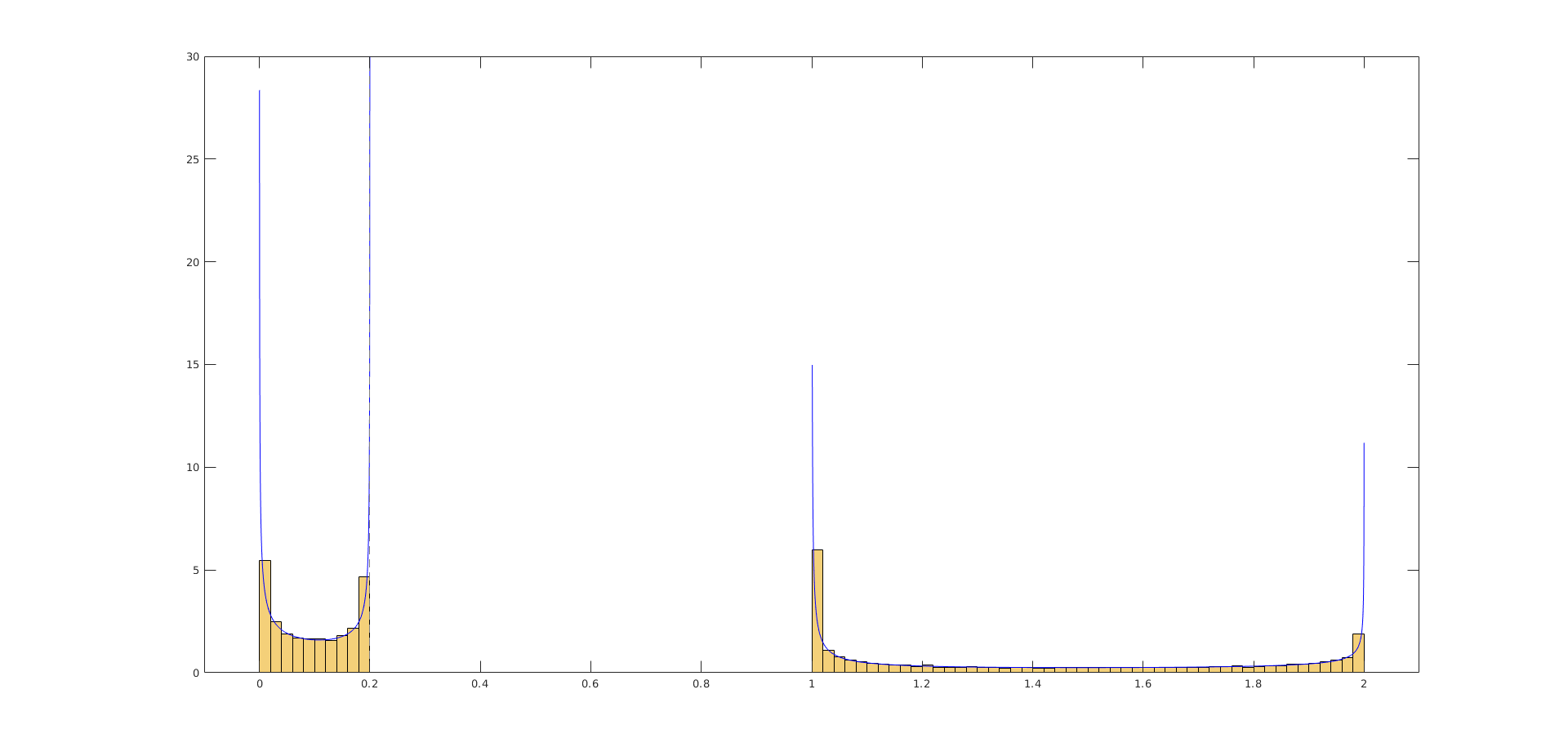}}
\caption{Histogram of eigenvalues of $P+QPQ$ when $P, Q$ are projections onto independent Haar distributed subspaces in $\gG_{2n,n}$ for $n=1000$, together with the limit distribution.}
\label{figure1}
\end{figure}
\end{example}

\medskip

More generally, our method applies to describe the distribution of a polynomial in two free elements whose distributions are supported on two points, since they are affine image of projections. Extending the method to distributions supported on three points seems out of reach.

\section*{Acknowledgements}
We thank the authors of \cite{FMNS} for fruitful discussions. The author was supported in part by ANR (France) under the grant ESQuisses (ANR-20-CE47-0014-01)

\bibliography{main}{}
\bibliographystyle{plain}

\end{document}